\documentclass[11pt, letterpaper]{article}    	
 \usepackage[margin=1in]{geometry}
\usepackage{setspace}
\usepackage{graphicx}	
\usepackage{subcaption}
\usepackage{float}

\usepackage[draft]{todonotes}   

\usepackage{enumerate}	
\usepackage{paralist}		

\usepackage{amssymb}
\usepackage{amsmath, amsthm, amssymb}

\newtheorem{thm}{Theorem}[section]

\newtheorem{prop}[thm]{Proposition}
\newtheorem{claim}{Claim}
\newtheorem{lemma}[thm]{Lemma}
\newtheorem{cor}[thm]{Corollary}
\newtheorem{prob}[thm]{Problem}

\newtheoremstyle{remark}%
    {8pt plus2pt minus4pt}%
    {8pt plus2pt minus4pt}%
    {\upshape}
    {}%
    {\bfseries\scshape}%
    {}%
    {6pt}
    {}%

\theoremstyle{remark}
\newtheorem{rem}[thm]{\textbf {Remark}}

\theoremstyle{definition}

\usepackage{tikz}
\newcounter{casenum}

\newcommand*{\rom}[1]{\expandafter{\romannumeral #1\relax}}

\usepackage{authblk}
\makeatletter
\def\@cite#1#2{{\normalfont[{\bfseries#1\if@tempswa , #2\fi}]}}
\makeatother

\title{An analogue of the Erd\H{o}s--Gallai theorem for random graphs}

\author{ J\'ozsef Balogh\thanks{Department of Mathematical Sciences, University of Illinois at Urbana-Champaign, IL, USA, and Moscow Institute of Physics and Technology, Russian Federation. Email: \texttt{jobal@illinois.edu}.
Partially supported by NSF Grant DMS-1764123 and Arnold O. Beckman Research Award (UIUC) Campus Research Board 18132 and the Langan Scholar Fund (UIUC).} \qquad Andrzej Dudek\thanks{Department of Mathematics, Western Michigan University, Kalamazoo, MI, USA. Email: \texttt{andrzej.dudek @wmich.edu}. Partially supported  by Simons Foundation Grant \#522400.} \qquad Lina Li\thanks{Department of Mathematics, University of Illinois at Urbana-Champaign, Urbana, IL, USA. Email: \texttt{linali2 @illinois.edu}.}}

\begin{document}

\maketitle

\begin{abstract}
Recently, variants of many classical extremal theorems have been proved in the random environment. We, complementing existing results, extend the Erd\H{o}s-Gallai Theorem in random graphs. In particular, we determine, up to a constant factor, the maximum number of edges in a $P_n$-free subgraph of $G(N,p)$, practically for all values of $N,n$ and $p$. Our work is also motivated by the recent progress on the size-Ramsey number of paths.
\end{abstract}

\section{Introduction}
A celebrated theorem of Erd\H{o}s and Gallai~\cite{EG1959} from 1959 determines the maximum number of edges in an $n$-vertex graph with no $k$-vertex path~$P_k$.
\begin{thm}[Erd\H{o}s and Gallai~\cite{EG1959}]\label{EG}
For $n, k\geq 2$, if $G$ is an $n$-vertex graph with no copy of~$P_k$, then the number of edges of $G$ satisfies $e(G)\leq \frac12 (k-2)n$. If $n$ is divisible by $k-1$, then the maximum is achieved by a union of disjoint copies of $K_{k-1}$. 
\end{thm}

An important direction of combinatorics in recent years is the study of sparse random analogues of classical extremal results; that is, the extent to which of these results remain true in a random setting.
For graphs $G$ and $F$, we write $\mathrm{ex}(G, F)$ for the maximum number of edges in an $F$-free subgraph of $G$. For example, the 
Erd\H{o}s--Gallai theorem asserts that $\mathrm{ex}(K_n, P_k) = \frac12 (k-2)n$ if $n$ is divisible by $k-1$.

The study of the random variable $\mathrm{ex}(G, F)$, where $G$ is the Erd\H{o}s-R\'{e}nyi random graph $G(n, p)$, was initiated by Babai, Simonovits and Spencer~\cite{babai1990extremal}, and by Frankl and R\"{o}dl~\cite{frankl1986large}. 
After efforts by several researchers~\cite{haxell1995turan, haxell1996turan, kohayakawa1998extremal, kohayakawa1997onk, kohayakawa2004turan, szabo2003turan}, Conlon and Gowers~\cite{conlon2016combinatorial} and Schacht~\cite{schacht2016extremal} finally proved a sparse random version of the Erd\H{o}s-Stone theorem, showing a \textit{transference principle} of Tur\'{a}n-type results, that is, when a random graph inherits its (relative) extremal properties from the classical deterministic case.
%
Note that via the hypergraph container method the same results were proved (\cite{BMS} and \cite{ST}), even when $|F|$ is a reasonable large function of $n$.
A special case of this result, when $F$ is the $k$-vertex path $P_k$, can be viewed as a weak analogue (as the Tur\'{a}n density is 0) of the Erd\H{o}s-Gallai theorem on the random graph for paths with a fixed size. In this paper, we investigate the random analogue of the Erd\H{o}s-Gallai theorem for general paths, whose length might increase with the order of the random graph.

We say that events $A_n$ in a probability space hold \textit{asymptotically almost surely} (or a.a.s.), if the probability that $A_n$ holds tends to 1 as $n$ goes to infinity. 
The typical appearance of long paths and cycles is one of the most thoroughly studied direction in random graph theory. Over the past decades, there were many diverse and beautiful results in this subject. In a seminal paper, Ajtai, Koml\'{o}s and Szemer\'{e}di~\cite{ajtai1981longest}, confirming a conjecture of Erd\H{o}s, proved that for $p=\frac{c}{n}$ with $c>1$, $G(n, p)$ contains a path of length $\alpha(c)n$ a.a.s.~where $\lim_{c\rightarrow \infty}\alpha(c)=1$. Frieze~\cite{frieze1986large} later determined the asymptotics of the number of vertices not covered by a longest path in $G(n,p)$. For Hamiltonicity, Bollob\'{a}s~\cite{bollobas1984evolution} and Koml\'{o}s and Szemer\'{e}di~\cite{komlos1983limit} independently proved that for $p\geq \frac{\log n + \log\log n +\omega(1)}{n}$, the random graph $G(n, p)$ is a.a.s.~Hamiltonian. 
Tur\'{a}n-type results for long cycles in $G(n, p)$ was also studied under the name of \textit{global resilience}, that is, the minimum number $r$ such that one can destroy the graph property by deleting $r$ edges. 
Dellamonica Jr, Kohayakawa, Marciniszyn and Steger~\cite{dellamonica2008resilience} determined the global resilience of $G(n, p)$ with respect to the property of containing a cycle of length proportional to the number of vertices.
Very recently, Krivelevich, Kronenberg and Mond~\cite{Krivelevich2019turan} studied the transference principle in the context of long cycles and in particular  showing the following.
\begin{thm}[Corollary 1.10 in~\cite{Krivelevich2019turan}]
For every $0<\beta<\frac{1}{5}$, there exists $C>0$ such that if $G=G(N,p)$ where $p\ge \frac{C}{N}$, then for any $\frac{C_1}{\log(1/\beta)}\cdot \log N \le n \le (1-C_2\beta)N$, with probability $1-e^{\Omega(N)}$,
\begin{equation}\label{eq: KKM}
\mathrm{ex}(G(N, p), C_n)\leq \left(\frac{\mathrm{ex}(K_N, C_n)}{\binom{N}{2}} + \beta \right)e(G(N, p)),
\end{equation}
where $C_1,C_2 > 0$ are absolute constants.
\end{thm}

We aim to explore the global resilience of general long paths. More formally, given integers $N>n$, we are interested in determining the asymptotic behavior of random variable $\mathrm{ex}(G(N, p), P_{n+1})$ as $N$ and $n$ go to infinity at the same time. 

We start with an observation, which is proved in Section~\ref{sec:proofs}.
\begin{prop}\label{prop:lower}
For every $\frac{1}{N^2}\ll p\leq \frac{1}{N}$ and $n\ge 2$, a.a.s.~we have $\mathrm{ex}(G(N, p), P_{n+1})=\Theta(pN^2)$. In particular, a.a.s.~$\mathrm{ex}(G(N, 1/N), P_{n+1}) \ge N/15$.
\end{prop}
\noindent
Therefore, throughout this paper, we naturally restrict ourselves to the regime $p \geq 1/N$ and have the following trivial lower bound
\begin{equation}\label{trilower}
a.a.s.\quad \mathrm{ex}(G(N, p), P_{n+1})\geq  \mathrm{ex}\left(G\left(N, 1/N\right), P_{n+1}\right)\geq N/15.
\end{equation}
We prove the following results.


\begin{thm}\label{thm: smallN}
Let $3n \leq N \leq ne^{2n}$. The following hold a.a.s.~as $n$ approaches infinity.
\begin{enumerate}[i)]
\item For $p\ge \left(\log\frac{N}{n}\right)/(6n)$, we have
$
\frac14pnN\leq \mathrm{ex}(G(N, p), P_{n+1})\leq 18pnN.
$

\item Let $\omega=\left(\log\frac{N}{n}\right)/(np)$. For $N^{-1} \leq p\leq \left(\log\frac{N}{n}\right)/(6n)$, we have 
\[
\frac{1}{75}\frac{\omega}{\log\omega}pnN\leq \mathrm{ex}(G(N, p), P_{n+1})\leq 8\frac{\omega}{\log\omega}pnN.
\]
\end{enumerate} 
\end{thm}

\begin{thm}\label{thm: largeN}
Let $N \geq ne^{2n}$. The following hold a.a.s.~as $n$ approaches infinity.
\begin{enumerate}[i)]
\item For $p\geq N^{-\frac{2}{5n}}$, we have
$
\frac{1}{16}nN\leq \mathrm{ex}(G(N, p), P_{n+1})\leq \frac12nN.
$

\item Let $\omega=\left(\log N\right)/(np)$. For $N^{-1} \leq p\leq N^{-\frac{2}{5n}}$, we have
\[
\frac{1}{75}\frac{\omega}{\log\omega}pnN\leq \mathrm{ex}(G(N, p), P_{n+1})\leq 8\frac{\omega}{\log\omega}pnN.
\]
\end{enumerate} 
\end{thm}

\begin{rem}
Assume that $n$ is even. Then \eqref{eq: KKM} together with $\mathrm{ex}(K_N, C_n) \le nN^{1 + 2/n}$~\cite{pikhurko2012}
imply that
\begin{equation*}
\begin{split}
\mathrm{ex}(G(N, p), P_n)&\leq \mathrm{ex}(G(N, p), C_n)
\leq \left(\frac{\mathrm{ex}(K_N, C_n)}{\binom{N}{2}} + \beta \right)e(G(N, p))\\
&\leq \left(\frac{nN^{1 + 2/n}}{\binom {N}{2}} + \beta \right)\frac{pN^2}{2}
\sim pnN^{1 + 2/n} + \beta \frac{pN^2}{2},
\end{split}
\end{equation*}
which is weaker than our bounds. (Recall that $p\ge \frac{C}{n}$, where $C=C(\beta)$.)
Of course, there are some better upper bounds for $\mathrm{ex}(K_N, C_n)$, which could be used to make an improvement. However, since, in general, $\mathrm{ex}(K_N, C_n)$ behaves differently with $\mathrm{ex}(K_N, P_n)$ and is indeed much greater, Krivelevich, Kronenberg, and Mond's result~\cite{Krivelevich2019turan} and ours do not imply one another.
\end{rem}

\begin{rem}
One can run the same proof and show that Theorem~\ref{thm: largeN} holds when $n$ is a constant greater than 1 and $N$ approaches infinity. Note also that a result of Johansson, Kahn and Vu~\cite{johansson2008factors} on the threshold function of the property that $G(N, p)$ contains a $K_n$-factor ($n$ is a constant) implies $\mathrm{ex}(G(N, p), P_{n+1})=\frac12(n-1)N$ for $p=\Omega\left(N^{-2/n}(\log n)^{1/\binom{n}{2}}\right)$, whenever $N$ is divisible by~$n$.
Indeed, they determined the threshold function for containing a $H$-factor ($H$ is a fixed graph), which might be useful for further improving the above result.
\end{rem} 
 
We made no attempt to optimize the constants in the theorems. Throughout the paper, we omit all floor and ceiling signs whenever these are not crucial. All logarithms in this paper have base $e$. 

\section{Tools}
In this section, we list several results that we will use. The first lemma is a direct application of the depth first search algorithm (DFS), which has appeared in~\cite{DP2018}. Using the DFS algorithm in finding long paths was first introduced by Ben-Eliezer, Krivelevich, and Sudakov~\cite{ben2012size}, and then it became a particularly suitable tool in this topic.

\begin{lemma}[\cite{DP2018}]\label{lemma: DFS}
For every $P_{n+1}$-free graph $H$ on $N$ vertices, we can find a decomposition of edges into $\bigcup_{i=1}^{N/n}F_i$, where $F_i=E(S_i)\cup E(S_i, T_i)$ for two disjoint sets $S_i, T_i\subseteq[N]$ with $|S_i|=|T_i|=n$. 
\end{lemma}

We also need the following form of Chernoff's bound.

\begin{lemma}[Chernoff's Bound]\label{chernoff}
Let $X=\sum_{i=1}^n X_i$, where $X_i=1$ with probability $p_i$ and $X_i=0$ with probability $1 - p_i$, and all $X_i$'s are independent. Let $\mu=\mathbb{E}(X)=\sum_{i=1}^np_i$. Then, for all $0<\delta<1$,
\[
\mathbb{P}(X \leq (1 - \delta)\mu) \leq e^{-\mu\delta^2/2}.
\]
\end{lemma}

The third lemma is a key ingredient of our proof, which is used to find dense subsets in random graphs. This may be of independent interest. 

\begin{lemma}\label{lemma: denseset}
For $N>2n$, $0<p< 1$ and a constant $0<\alpha\leq 1/2$, let $r=N/n$ and choose an arbitrary $\beta$ satisfying
\begin{equation}\label{eq:beta}
\max\left\{2\log(2e),\ \frac{2}{\alpha np}\log\left(\frac{1}{\alpha np}\right)\right\} \leq 2\beta\log \beta \le \min\left\{2\left(\frac{1}{p}\right)\log \left(\frac{1}{p}\right),\ \frac{1}{np}\left(\log r - \log\alpha 2^{\frac{1}{\alpha}}\right)\right\}.
\end{equation}
Then there exists a positive constant $c=c(\alpha)$ such that with probability at least $1-\exp(-c r^\alpha n)$ there exists an $n$-set in $G(N, p)$ with at least $\left(\frac{1 - \alpha}{2}\right)\beta pn^2$ edges.
\end{lemma}

\begin{rem}
Lemma~\ref{lemma: denseset} essentially states that given $N, n$, for some range of $p$, we can find an $n$-vertex subgraph, which is denser than the random graph by some factor $\beta$. For instance, as it will be explained in the proof of Theorem~\ref{thm: smallN}~(ii),
when $135n\leq N\leq ne^{2n}$, we can choose $\frac{\log r}{n r^{1/5}}\leq p\leq \frac{\log r}{6n}$, so that $2\beta\log \beta=\frac{1}{np}\log\left(\frac{3}{8}r\right)$ satisfying~(\ref{eq:beta}). Note that if $p\ll \frac{\log r}{n}$, we have $\beta=\omega(1)$, and therefore the graph we produce here is much denser than the random graph.
\end{rem}
\begin{proof}
One can check that the function $f(x)=x \log x$ is non-negative and increasing for $x\ge 1$. Thus, $\log(2e)\leq f(\beta) \le f(1/p)$ implies that 
\begin{equation}\label{betarange}
\max\left\{2, \frac{1}{\alpha np}\right\}<\beta \le 1/p.
\end{equation}

Let $B_0=[N]$. We will construct the desired set iteratively. In each step, take an arbitrary subset $A_i\subseteq B_{i-1}$ of size $\alpha n$, and let
\[
B_i=\{v\in B_{i-1}\setminus A_i:\ \mathrm{deg}(v, A_i)\geq \beta\alpha np\}.
\]
We will show that a.a.s.~we can continue this process $\lceil\frac{1}{\alpha}\rceil$ steps. For convenience, in the rest of the proof, we ignore all floor and ceiling signs.
\begin{claim}\label{claim: 2}
$|B_i|\geq \frac{rn}{2^i}\exp\left(-2i\beta\log\beta\cdot \alpha np\right)$, for all $0\leq i\leq \frac{1}{\alpha} -1$ with probability at least $1-\exp(-\Omega(r^\alpha n))$.
\end{claim}
We prove it by induction on $i\ge 0$. For $i=0$, it is trivial. Suppose the statement holds for $i-1$. That means
\begin{equation}\label{eq:lem:ind}
|B_{i-1}| \ge \frac{rn}{2^{i-1}}\exp\left(-2(i-1)\beta\log\beta\cdot \alpha np\right)
\end{equation}
with probability at least $1-\exp(-\Omega(r^\alpha n))$. Furthermore, $0\leq i\leq \frac{1}{\alpha} -1$ yields that $(i-1)\alpha < i\alpha \le 1-\alpha < 1$ and hence,
\[
|B_{i-1}|\geq \frac{rn}{2^{\frac{1}{\alpha} -2}}\exp\left(-2\beta\log\beta\cdot np\right)
\geq\frac{rn}{2^{\frac{1}{\alpha} -2}}\exp\left(-\left(\log r - \log\alpha 2^{\frac{1}{\alpha}}\right)\right)=4\alpha n,
\]
consequently
\[
|B_{i-1}| - \alpha n \ge \frac{3}{4} |B_{i-1}| > \frac{|B_{i-1}|}{\sqrt{2}}.
\]
Then, the expected size of $B_i$ is 
\[
\mathbb{E}(|B_i|)=(|B_{i-1}| - \alpha n)\mathbb{P}(\mathrm{deg}(v, A_i)\geq \beta \alpha np)
\geq \frac{1}{\sqrt{2}}|B_{i-1}|\binom{\alpha n}{\beta \alpha np}p^{\beta \alpha np}(1-p)^{\alpha n}.
\]
Due to (\ref{betarange}), we get that $p \le 1/\beta \le 1/2$ and $\beta \alpha np\geq 1$. Now we use $\binom{\alpha n}{\beta \alpha np} \ge \left( \frac{\alpha n}{\beta \alpha np} \right)^{\beta \alpha np} = 
\left( \frac{1}{\beta p} \right)^{\beta \alpha np}$ and the inequality $1-p \ge (2 e)^{-p}$, which is valid for $0 \le p \le 1/2$. 
Thus, 
\begin{align*}
\mathbb{E}(|B_i|)&=(|B_{i-1}| - \alpha n)\mathbb{P}(\mathrm{deg}(v, A_i)\geq \beta \alpha np)
\geq \frac{1}{\sqrt{2}}|B_{i-1}|\exp(-(\beta\log\beta + \log 2e)\alpha np)\\
&\geq \frac{1}{\sqrt{2}}|B_{i-1}|\exp(-2\beta\log\beta \cdot \alpha np).
\end{align*}
Observe that conditioning on~\eqref{eq:lem:ind} gives 
\begin{align*}
\mathbb{E}(|B_i|) &\ge \frac{1}{\sqrt{2}}|B_{i-1}|\exp(-2\beta\log\beta \cdot \alpha np)\\
&\ge \frac{1}{\sqrt{2}}   \cdot \frac{rn}{2^{i-1}}\exp\left(-2(i-1)\beta\log\beta\cdot \alpha np\right) \cdot \exp(-2\beta\log\beta \cdot \alpha np)\\
&= \frac{1}{\sqrt{2}}   \cdot \frac{rn}{2^{i-1}} \exp\left(-2i\beta\log\beta\cdot \alpha np\right)
\geq \frac{1}{\sqrt{2}}   \cdot \frac{rn}{2^{i-1}} \exp\left(-\alpha i \left(\log r - \log\alpha 2^{\frac{1}{\alpha}}\right)\right)\\
&\ge \frac{1}{\sqrt{2}}   \cdot \frac{rn}{2^{\frac{1}{\alpha}-1}} \exp\left(-(1-\alpha) \left(\log r - \log\alpha 2^{\frac{1}{\alpha}}\right)\right)
= \Omega(r^\alpha n),
\end{align*}
which goes to infinity together with $n$. Therefore, Chernoff's bound (applied with $\delta=1-1/\sqrt{2}$) yields that with probability at least $1-\exp(-\Omega(r^\alpha n))$ we have
\[
|B_i|\geq \frac{1}{\sqrt{2}}\mathbb{E}(|B_i|)\geq\frac{1}{2}|B_{i-1}|\exp(-2\beta\log\beta \cdot \alpha np)\geq \frac{rn}{2^i}\exp\left(-2i\beta\log\beta \cdot \alpha np\right),
\]
where the last inequality follows from~\eqref{eq:lem:ind}. 
\newline

Now we finish the proof of Lemma~\ref{lemma: denseset}.
Claim~\ref{claim: 2} gives that with probability at least $1-\exp(-\Omega(r^\alpha n))$ the set $B_{\frac{1}{\alpha} -1}$ exists and satisfies
\[
\left|B_{\frac{1}{\alpha} -1}\right|\geq \frac{rn}{2^{\frac{1}{\alpha} -1}}\exp\left(-\left(\log r - \log\alpha 2^{\frac{1}{\alpha}}\right)\right)=2\alpha n>\alpha n.
\]
Therefore, we can find disjoint sets $A_1, \ldots, A_{1/\alpha}$ of size $\alpha n$ with $e(A_i,A_j)\geq \alpha n\cdot \beta \alpha np$ for all $1\leq i<j\leq 1/\alpha$.
Let $A=\bigcup_{i=1}^{1/\alpha} A_i$. Then we have $|A|=n$ and
\[
e(A)\geq \binom{1/\alpha}{2}\alpha n\cdot \beta \alpha np=\left(\frac{1 - \alpha}{2}\right)\beta pn^2.
\]
\end{proof} 

We also present the following two probabilistic results which will be used later.
\begin{lemma}\label{claim: upper1}
Assume that $np\geq \left(\log \frac{N}{n}\right)/6$ and $N\geq 3n$. Then a.a.s.~for every two disjoint sets $S, T\subseteq[N]$, $|S|=|T|=n$, the number of edges in $G\in G(N, p)$ induced by $S\cup T$ with at least one endpoint in $S$ is at most $18n^2p$. 
\end{lemma}
\begin{proof}
Let $X_{S, T}$ be the number of edges in $G(N, p)$ with one endpoint in $S$ and one endpoint in $T$. Observe that $\mathbb{E}(X_{S, T})=\left(\frac32 - \frac1{2n}\right)n^2p$. Note that if $3n^2/2 \le 18n^2p$, then the statement is trivial. Otherwise, 
the union bound implies that
\[
\begin{split}
\mathbb{P}(\exists S, T,  X_{S, T}\geq 18n^2 p)&\leq \binom{N}{n}^2 \binom{3n^2/2}{18n^2p}p^{18n^2p}
\leq\left(\frac{Ne}{n}\right)^{2n}\left(\frac{e}{12}\right)^{18n^2p}\\
&=\exp\left(-n\left(18np\log\left(\frac{12}{e}\right) - 2\log\left( \frac{Ne}{n}\right)\right)\right).
\end{split}
\]
Since $np\geq \left(\log \frac{N}{n}\right)/6$ and $N\geq 3n$, we obtain that
\begin{multline*}
18np\log\left(\frac{12}{e}\right) - 2\log\left( \frac{Ne}{n} \right)\geq
3\log\left( \frac{12}{e}\right) \log \left( \frac{N}{n} \right) - 2\log\left( \frac{Ne}{n} \right)\geq\\
 4\log\left( \frac{N}{n}\right)-2\log \left(\frac{N}{n}\right)-2=2\log \left(\frac{N}{n}\right)-2\geq 2\log 3-2\geq 0.19.
\end{multline*}
Finally, we conclude that $\mathbb{P}(\exists S, T,  X_{S, T}\geq 18n^2 p)\leq \exp(-0.19 n)=o(1)$, which completes the proof.
\end{proof}

\begin{lemma}\label{claim: upper2}
Let $\beta=\frac{\frac{1}{np}\log\frac{N}{n}}{\log\left(\frac{1}{np}\log\frac{N}{n} \right)}>1$ and $m=8\beta n^2p$. Then a.a.s.~for every two disjoint sets $S, T\subseteq[N]$, $|S|=|T|=n$, the number of edges induced by $S\cup T$ with at least one endpoint in $S$ is at most $m$. 
\end{lemma}
\begin{proof}
We assume $m< 3n^2/2$ since otherwise Lemma~\ref{claim: upper2} holds trivially. By a simple union bound, we obtain
\begin{align*}
\mathbb{P}(\exists S, T,  X_{S, T}\geq m)&\leq \binom{N}{n}^2 \binom{3n^2/2}{m}p^m
\leq\exp\left(2n\log \left( \frac{Ne}{n} \right) \right)\exp(-\log\beta \cdot m)\\
&= \exp\left(2n\log \left( \frac{Ne}{n} \right)-8\beta\log\beta \cdot n^2p\right).
\end{align*}
Now we bound from below $\beta\log\beta$ by
\begin{align*}
\beta\log\beta = \frac{\frac{1}{np}\log\frac{N}{n}}{\log\left(\frac{1}{np}\log\frac{N}{n} \right)}
  \log\left(\frac{\frac{1}{np}\log\frac{N}{n}}{\log\left(\frac{1}{np}\log\frac{N}{n} \right)} \right)
\ge \frac{\frac{1}{np}\log\frac{N}{n}}{\log\left(\frac{1}{np}\log\frac{N}{n} \right)} 
 {\log\sqrt{\frac{1}{np}\log\frac{N}{n}}}
= \frac{1}{2np}\log\left(\frac{N}{n} \right).
\end{align*}
Thus,
\begin{align*}
\mathbb{P}(\exists S, T,  X_{S, T}\geq m)&\le \exp\left(2n\log\left( \frac{Ne}{n}\right)-8\beta\log\beta \cdot n^2p\right)\\
&\leq \exp\left(2n\log\left( \frac{Ne}{n} \right)-\frac{4}{np}\log\left( \frac{N}{n} \right) \cdot n^2p\right)\\
&\leq\exp\left(-n\left(4\log\left(\frac{N}{n}\right)-2\log\left( \frac{Ne}{n} \right)\right)\right)=o(1),
\end{align*}
where the last inequality follows from $N\geq 3n$ as
$4\log\left( \frac{N}{n} \right)-2\log\left( \frac{Ne}{n} \right) = 2\log\left( \frac{N}{n} \right)-2 \ge 2 \log(3)-2 \ge 0.19.$
\end{proof}

\section{Proofs of the main results}\label{sec:proofs}

\subsection{Proof of Proposition~\ref{prop:lower}}
Let $G = (V,E) = G(N, p)$. 
We will count the number of isolated edges. For a given pair of vertices $e\in \binom{V}{2}$, let $X_e$ be an indicator random variable that takes value $1$ if $e$ is an isolated edge in~$G$. Set $X = \sum_{e}X_e$. Observe that $\Pr(X_e=1) = p(1-p)^{2(N-2)}$ and so
\[
\mathbb{E}(X) = \binom{N}{2} p(1-p)^{2(N-2)} \sim \binom{N}{2} p e^{-2pN} \ge \binom{N}{2} p e^{-2} \to\infty,
\]
by assumption.
Furthermore, since for vertex disjoint $e, f \in \binom{V}{2}$, $\Pr(X_e = X_f = 1) = p^2 (1-p)^{4(n-4)+4}$, we obtain that
\[
\mathbb{E}(X^2) = \mathbb{E}(X) + \sum_{e\cap f=\emptyset} \Pr(X_e = X_f = 1)
= \mathbb{E}(X) + 6\binom{N}{4}p^2 (1-p)^{4(N-4)+4}.
\]
Thus, 
\[
\frac{\mathbb{E}(X^2)}{\mathbb{E}(X)^2} = \frac{1}{\mathbb{E}(X)} + 
\frac{(N-2)(N-3)}{N(N-1)(1-p)^4}
\le \frac{1}{\mathbb{E}(X)} + \frac{1}{(1-p)^4}
\le \frac{1}{\mathbb{E}(X)} + \frac{1}{1-4p}
\]
and
\[
\frac{\mathrm{Var}(X)}{\mathbb{E}(X^2)} \le \frac{1}{\mathbb{E}(X)} + \frac{1}{1-4p} - 1 = \frac{1}{\mathbb{E}(X)} + \frac{4p}{1-4p}
= o(1),
\]
since $\mathbb{E}(X)\to\infty$ and also by assumption $p\to0$. Now Chebyshev's inequality yields that $X$ is concentrated around its mean and consequently a.a.s.~we have
\[
\mathrm{ex}(G(N, p), P_{n+1}) \ge (1+o(1))\mathbb{E}(X) = \Omega(pN^2).
\]
The upper bound easily follows from the fact that $\mathrm{ex}(G(N, p), P_{n+1})\leq e(G(N, p))$.

Finally observe that a.a.s.
\[
\mathrm{ex}(G(N, 1/N), P_{n+1}) \ge (1+o(1))\mathbb{E}(X) \ge (1+o(1))\binom{N}{2} \frac{1}{N} e^{-2}   \ge  N/15.
\]
\qed

\subsection{Proof of Theorem~\ref{thm: smallN}}
\textit{Proof of Theorem~\ref{thm: smallN}~(i).} This proof is by now quite standard which applies the DFS algorithm and the first moment method. Recall that $np\geq \left(\log \frac{N}{n}\right)/6$ and $N\geq 3n$. 

Observe that Lemma~\ref{claim: upper1} together with Lemma~\ref{lemma: DFS} imply that for every $P_{n+1}$-free subgraph $H$ of $G\in G(N, p)$ a.a.s.
\[
e(H)\leq \frac{N}{n}\cdot 18n^2p=18pnN,
\]
which establishes the upper bound.

For the lower bound, take an arbitrary vertex partition $[N]=\bigcup_{i=1}^{N/n}S_i$, where $|S_i|=n$ for all~$i$. Let $H$ be the subgraph of $G\in G(N, p)$ whose edge set is $\bigcup E(G[S_i])$. Clearly, $H$ is $P_{n+1}$-free. Note that $\mathbb{E}(e(H))=\frac{N}{n}\left(\frac12 - \frac1{2n}\right)n^2p=\left(\frac12 - \frac1{2n}\right)pnN$. By Chernoff's bound, 
\[
\mathbb{P}\left(e(H)\leq \frac14pnN\right)\leq \exp\left(- \Omega(pnN)\right)=o(1),
\]
since $pnN\to\infty$.
Therefore, a.a.s.~we have $\mathrm{ex}(G(N, p), P_{n+1})\geq e(H)\geq  \frac14pnN$.
\qed
\newline
\newline
\textit{Proof of Theorem~\ref{thm: smallN}~(ii).} We first show the upper bound.
Let $\beta_1=\frac{\frac{1}{np}\log\frac{N}{n}}{\log\left(\frac{1}{np}\log\frac{N}{n} \right)}$ and $m=8\beta_1 n^2p$. Since $np\leq \left(\log \frac{N}{n}\right)/6$, we know that $\beta_1>1$.

For every $P_{n+1}$-free subgraph $H$ of $G\in G(N, p)$, Lemma~\ref{lemma: DFS} and Lemma~\ref{claim: upper2} imply that a.a.s 
\[
e(H)\leq \frac{N}{n}\cdot m=8\beta_1 pnN=8\frac{\frac{1}{np}\log\frac{N}{n}}{\log\left(\frac{1}{np}\log\frac{N}{n} \right)} pnN,
\]
which establishes the upper bound.
\newline

For the lower bound, we shall divide the discussion into three cases. First, let us assume $N\leq 135n$. Together with $\frac{1}{np}\log\left(\frac{N}{n}\right)\geq 6\geq e$, we have
\[
\frac{\omega}{\log\omega}pnN = \frac{\log \left(\frac N n\right)}{\log\left(\frac{1}{np}\log\left(\frac{N}{n} \right)\right)}N
\le \log \left(\frac{N}{n}\right)N < 5N.
\]
Therefore, by (\ref{trilower}), we trivially have
\[
\mathrm{ex}(G(N, p), P_{n+1})\geq N/15 \geq \frac{1}{75}\frac{\omega}{\log\omega}pnN.
\]
Next, let us assume $p\le \log\left(\frac{N}{n}\right)/\left(n\left(\frac{N}{n}\right)^{1/5}\right)$.
Similarly, we complete the proof by observing that 
\[
\frac{\omega}{\log\omega}pnN = \frac{\log \left(\frac N n\right)}{\log\left(\frac{1}{np}\log\left( \frac{N}{n} \right)\right)}N
\le \frac{\log \left(\frac{N}{n}\right)}{\frac{1}{5}\log \left(\frac{N}{n}\right)}N =5N.
\]

It remains to prove the lower bound for the case when $N\geq 135n$ and
\begin{equation}\label{eq: prob}
\frac{ \log\left(\frac{N}{n}\right) }{ n\left(\frac{N}{n}\right)^{1/5}} \leq p\leq \frac{ \log(\frac{N}{n})}{6n}.
\end{equation}
Indeed, such range of $p$ only exists for $N\geq 6^5n$.
In this case, we will apply Lemma~\ref{lemma: denseset} repeatedly to find a dense subgraph with no $P_{n+1}$. Let 
\[
2\beta_2\log \beta_2=\min\left\{2\left(\frac{1}{p}\right)\log \left(\frac{1}{p}\right),\ \frac{1}{np}\log \left(\frac {3N}{8n}\right)\right\}.
\]
Since $N \leq ne^{2n}$ and $p\leq \log\left(\frac{N}{n}\right)/(6n)\leq \frac{1}{3}$, we have 
\[
2\left(\frac{1}{p}\right)\log \left(\frac{1}{p}\right)
\geq 2 \left(\frac{1}{p}\right) \log 3
> \frac{2}{p}
\geq \frac{1}{np}\log \left(\frac{3N}{8n}\right).
\]
Furthermore, since $N\ge 6^5n$, we obtain
\[
\log \left( \frac{3N}{8n} \right) 
\ge \log \left( \frac{3}{8} \right)+ \frac{1}{5} \log 6^5 + \frac{4}{5} \log \left( \frac{N}{n}\right) > \frac{4}{5} \log \left(\frac{N}{n}\right),
\]
and
\[
2\beta_2\log \beta_2 = \frac{1}{np}\log \left(\frac {3N}{8n}\right)\geq \frac{4}{5np}\log \left(\frac{N}{n}\right)>2\log(2e).
\]
Finally, observe that for $\alpha = 1/2$,
\[
\frac{1}{np} \log \left( \frac{3N}{8n} \right) 
\geq \frac{1}{np} \cdot 4 \log \left( \frac{2\left(\frac{N}{n}\right)^{1/5}}{\log\left(\frac{N}{n}\right)}\right)
\geq\frac{2}{\alpha np}\log\left(\frac{1}{\alpha np}\right),
\]
where the first inequality is given by $N\geq 135n$ and the last inequality follows from (\ref{eq: prob}).
Thus, we can iteratively apply Lemma~\ref{lemma: denseset} $N/4n$ times with $\alpha=\frac{1}{2}$ and $r=\frac{3N}{4n}$ and find $N/4n$ disjoint $n$-sets $A_i$, where a.a.s.~for all $i$
\[
e(A_i)\geq \left(\frac{1 - \alpha}{2}\right)\beta_2 pn^2
\geq \frac{1 - \alpha}{4}\frac{\frac{1}{np}\log\left( \frac{3N}{8n} \right)}{\log\left(\frac{1}{np}\log \left( \frac{3N}{8n} \right) \right)}pn^2
\geq \frac{1}{10}\frac{\frac{1}{np}\log\left( \frac{N}{n} \right)}{\log\left(\frac{1}{np}\log\left(\frac{N}{n}\right)\right)}pn^2.
\]

 Let $H$ be the subgraph of $G$ with vertex set $\bigcup_{i=1}^{N/4n}A_i$, and edge set $\bigcup_{i=1}^{N/4n} E(A_i)$. Note that $H$ is $P_{n+1}$-free and therefore, a.a.s.~we have 
\[
\mathrm{ex}(G(N, p), P_{n+1})\geq e(H)\geq \frac{1}{10}\frac{\frac{1}{np}\log\left( \frac{N}{n} \right)}{\log\left(\frac{1}{np}\log \left(\frac{N}{n}\right)\right)}pn^2\cdot \frac{N}{4n}
=\frac{1}{40}\frac{\frac{1}{np}\log\left( \frac{N}{n}\right)}{\log\left(\frac{1}{np}\log\left( \frac{N}{n} \right)\right)}pnN.
\]
\qed

\subsection{Proof of Theorem~\ref{thm: largeN}}
\textit{Proof of Theorem~\ref{thm: largeN}~(i).} By the the Erd\H{o}s-Gallai Theorem (Theorem~\ref{EG}), it is sufficient to prove the lower bound. Let
\[
2\beta\log \beta=\min\left\{2\left(\frac{1}{p}\right)\log \left(\frac{1}{p}\right),\ \frac{4}{5np}\log N\right\}.
\]
Since $p\geq N^{-\frac{2}{5n}}$, we have $\beta=1/p$. If $p> 1/3$, then the proof simply follows from the proof of Theorem~\ref{thm: smallN}~(i).
Otherwise, we have $2\beta\log \beta\geq 6\log 3> 2\log(2e)$. Similarly as in the proof of Theorem~\ref{thm: smallN}~(ii), we can iteratively apply Lemma~\ref{lemma: denseset} $N/4n$ times with $\alpha=\frac{1}{2}$ and $r=\frac{3N}{4n}$, and a.a.s.~find a $P_{n+1}$-free subgraph $H$ of $G(N, p)$ with
\[
e(H)\geq \left(\frac{1 - \alpha}{2}\right)\beta pn^2\cdot \frac{N}{4n}=\frac{1}{16}nN.
\]
\qed
\newline
\newline
\textit{Proof of Theorem~\ref{thm: largeN}~(ii).} The proof of the upper bound is the same as in Theorem~\ref{thm: smallN}~(ii) and we skip here the full details. For the lower bound, we first assume that $p<N^{-1/5}$. Observe that 
\[
\frac{\omega}{\log\omega}pnN = \frac{\log N}{\log\left(\frac{1}{np}\log N\right)}N
\le \frac{\log N}{\log N^{1/5}} N =5N,
\]
where the inequality holds for $N\geq n e^{2n}$.
Therefore, by (\ref{trilower}), we trivially have
\[
\mathrm{ex}(G(N, p), P_{n+1})\geq N/15 \geq \frac{1}{75}\frac{\omega}{\log\omega}pnN.
\]

It remains to show the lower bound for $p\geq N^{-1/5}$. Let
\[
2\beta\log \beta=\min\left\{2\left(\frac{1}{p}\right)\log \left(\frac{1}{p}\right),\ \frac{4}{5np}\log N\right\}.
\]
Since $p\leq N^{-\frac{2}{5n}}$, we have $2\beta\log \beta=\frac{4}{5np}\log N$.
Since $N\geq ne^{2n}$, we  have
\[
\frac{1}{np}\log \left(\frac {3N}{8n}\right)\ge 2\beta\log \beta
\geq \frac{4}{5np}\log \left(ne^{2n}\right)
\geq \frac{8}{5p}
\geq \frac{8}{5} N^{\frac{2}{5n}}
\geq \frac{8e^{\frac{4}{5}}}{5}
>2\log(2e).
\]
Moreover, observe that for $\alpha=\frac{1}{2}$ and $p\geq N^{-1/5}$, we have $2\beta\log \beta\geq \frac{2}{\alpha np}\log\left(\frac{1}{\alpha np}\right)$.
Similarly as in the proof of Theorem~\ref{thm: smallN}~(ii), the proof is completed by iteratively applying Lemma~\ref{lemma: denseset} $N/4n$ times with $\alpha=\frac{1}{2}$ and $r=\frac{3N}{4n}$.
\qed

\section{Long paths and multicolor size-Ramsey number}
The size-Ramsey number $\hat{R}(F, r)$ of a graph $F$ is the smallest integer $m$ such that there exists a graph $G$ on $m$ edges with the property that any $r$-coloring of the edges of $G$ yields a monochromatic copy of $F$. The study of size-Ramsey number was initiated by Erd\H{o}s, Faudree, Rousseau and Schelp~\cite{erdHos1978size}. For paths, Beck~\cite{beck1983size}, resolving a \$100 question of Erd\H{o}s, proved that $\hat{R}(P_n, 2)< 900n$ for sufficiently large $n$. The strongest upper bound, $\hat{R}(P_n, 2)\leq 74n$, was given by Dudek and Pra{\l}at~\cite{dudek2017some}, and they also provide the lower bound, $\hat{R}(P_n, 2)\geq 5n/2 - O(1)$. Very recently, Bal and DeBiasio~\cite{bal2019new} further improved the lower bound to $(3.75 - o(1))n$.

For more colors, it was proved in~\cite{dudek2017some} that $\frac{(r+3)r}{4}n - O(r^2)\leq \hat{R}(P_n, r)\leq 33r4^rn$.  
Subsequently, Krivelevich~\cite{krivelevich2019long} (see also~\cite{krivelevich2019expanders}) showed that $\hat{R}(P_n, r)=O((\log r)r^2 n)$. An alternative proof of the above result was later given by Dudek and Pra{\l}at~\cite{DP2018}.
Both proofs indeed give a stronger \textit{density-type} result, which shows that any dense subset of a large enough structure contain the desired substructure. In particular, the proof in~\cite{DP2018} implies the following result.

\begin{thm}[\cite{DP2018}]\label{thm: density}
For $r\geq 2$ and $c\geq 7$, there exists a constant $\alpha=\alpha(c)$ such that the following statement holds a.a.s.~for $p\geq \alpha(\log r)/n$. Every subgraph $H$ of $G\in G(crn, p)$ with $e(H)\geq e(G)/r$ contains a $P_{n+1}$.
\end{thm}

Note that any improvement of the order of magnitude of $p$ in the above theorem would improve the upper bound for $\hat{R}(P_n, r)$. However, Theorem~\ref{thm: smallN}~(ii) implies that when $p\ll \left(\log cr\right)/(6n)$, i.e. $(\log cr)/np\gg 6$, a.a.s.~there exists a $P_{n+1}$-free subgraph of $G\in G(crn, p)$ which contains more than
\[
\frac{1}{40}\frac{(\log cr)/np}{\log\left((\log cr)/np \right)}pn\cdot crn\geq cpn\cdot crn> e(G)/r
\] 
edges. Therefore, $(\log r)/n$ is the threshold function for the density statement in Theorem~\ref{thm: density}. It would be interesting to know if $(\log r)/n$ is still the threshold function for the corresponding Ramsey-type statement.

\section{Concluding remarks}
Our investigation raises some open problems. The most interesting question is to investigate the corresponding Ramsey properties on random graphs. The Ramsey-type questions on sparse random graphs has been studied by several researchers, for example, see~\cite{bohman2011ramsey, spohel2010coloring}.
\begin{prob}
Determine the threshold function $p(n)$ for the following statement. For some constant $c$ and $r\geq 2$ ($c$ is independent of $r$), every $r$-coloring of $G(crn, p)$ contains a monochromatic $P_{n+1}$.
\end{prob}
Theorem~\ref{thm: density} implies that $p(n)=O((\log r)/n)$, while the lower bound of $\hat{R}(P_n, r)$ shows that $p(n)=\Omega(1/n)$, where $n$ goes to infinity. The exact behavior of $p(n)$ remains open and its determination would be very useful for studying the size-Ramsey number of paths. 

Another direction is to consider the following graph parameter. Denote by $c(G, F)$ the minimum number of colors $k$ such that there exists a $k$-coloring of $G$ without monochromatic $F$. 
Clearly, we have 
\begin{equation}\label{eq: ramsey}
c(G(N, p), P_{n+1})
\geq \frac{\binom{N}{2}p}{\mathrm{ex}(G(N, p), P_{n+1})}
\geq \frac{pN^2}{3\mathrm{ex}(G(N, p), P_{n+1})}.
\end{equation}
Let $r=N/n$. We first present two general upper bounds on $c(G(N, p), P_{n+1})$.
\begin{thm}\label{thm: cupper}
Suppose $r$ is a prime power, then $c(G(N, p), P_{n+1})\leq r+1$.
\end{thm}
\begin{proof}
We use a construction from~\cite{gyarfas1977partition} (also appeared in~\cite{krivelevich2019long}). Let $A_r$ be an affine plane of order $r$, i.e. $r^2$ points with $r^2+r$ lines, where every pair of points is contained in a unique line, and the lines can be split into $r+1$ disjoint families $F_1, \ldots, F_{r+1}$ so that the lines inside the families are parallel.

We  arbitrarily partition $[N]$ into $r^2$ parts $V_1, V_2, \ldots, V_{r^2}$, where each part has size $N/r^2=n/r$. We define an $r+1$-coloring as follows. If $e$ is an edge crossing between $V_x$ and $V_y$, where the unique line containing $xy$ is in the family $F_i$, then we color $e$ by $i$. Observe that every connected subgraph in color $i$ has its vertex set $V$ inside $\cup_{x\in L}V_x$ for some line $L\in A_r$. Therefore, we have $|V|\leq r\cdot n/r=n$, and there is no monochromatic $P_{n+1}$.
\end{proof}

\begin{thm}
A.a.s.~$c(G(N, p), P_{n+1})\leq 2pN$.
\end{thm}
\begin{proof}
Let $k=2pN$, and we can assume $k\leq r+1$.
Consider a random $k$-coloring of $G(N, p)$. Then the subgraph $G_i$, whose edges are all edges in color $i$, is in $G(N, p')$, where $p'=p/k=1/2N$. A fundamental result of Erd\H{o}s and R\'{e}nyi shows that a.a.s the largest component of $G_i$ has size $O(\log N)\leq n$. Therefore, a.a.s.~there is no monochromatic $P_{n+1}$.
\end{proof}
\begin{cor}
If $p=\frac{1}{\omega \cdot n}$, where $\omega=\omega(r)\geq 2$, then a.a.s.~$c(G(N, p), P_{n+1})\leq 2r/\omega$.
\end{cor}
For the lower bound, the proof of Theorem 1.2. in~\cite{DP2018} implies the following.
\begin{thm}
For $p\geq 22(\log (r/7))/n$, a.a.s.~$c(G(N, p), P_{n+1})> r/7$.
\end{thm}
This together with Theorem~\ref{thm: cupper} shows that a.a.s.~$c(G(N, p), P_{n+1})=\Theta(r)$ for $p=\Omega((\log r)/n)$.
On the other hand, Theorem~\ref{thm: smallN} and~(\ref{eq: ramsey}) give a lower bound for small $p$.
\begin{thm}
For $p\leq (\log r)/34n$, a.a.s.
$c(G(N, p), P_{n+1})\geq\frac{\log \omega}{24\omega}r,$ where $\omega=(\log r)/np$. 
\end{thm}
This naturally raises the following question.
\begin{prob}
What is the exact behavior of $c(G(N, p), P_{n+1})$ for $p=o((\log r)/n)$, where $n$ goes to infinity?
\end{prob}

\textbf{Acknowledgment:}
We would like to thank to the referees for their valuable comments and suggestions.


\begin{thebibliography}{10}

\bibitem{ajtai1981longest}
{\sc M.~Ajtai, J.~Koml{\'o}s, and E.~Szemer{\'e}di}, {\em The longest path in a
  random graph}, Combinatorica, 1 (1981), pp.~1--12.

\bibitem{babai1990extremal}
{\sc L.~Babai, M.~Simonovits, and J.~Spencer}, {\em Extremal subgraphs of
  random graphs}, Journal of Graph Theory, 14 (1990), pp.~599--622.

\bibitem{bal2019new}
{\sc D.~Bal and L.~DeBiasio}, {\em New lower bounds on the size-ramsey number
  of a path}, arXiv preprint arXiv:1909.06354,  (2019).

\bibitem{BMS}
{\sc J.~Balogh, R.~Morris, and W.~Samotij}, {\em Independent sets in
  hypergraphs}, Journal of the American Mathematical Society, 28 (2015),
  pp.~669--709.

\bibitem{beck1983size}
{\sc J.~Beck}, {\em On size {R}amsey number of paths, trees, and circuits.
  {I}}, Journal of Graph Theory, 7 (1983), pp.~115--129.

\bibitem{ben2012size}
{\sc I.~Ben-Eliezer, M.~Krivelevich, and B.~Sudakov}, {\em The size {R}amsey
  number of a directed path}, Journal of Combinatorial Theory, Series B, 102
  (2012), pp.~743--755.

\bibitem{bohman2011ramsey}
{\sc T.~Bohman, A.~Frieze, M.~Krivelevich, P.-S. Loh, and B.~Sudakov}, {\em
  Ramsey games with giants}, Random Structures \& Algorithms, 38 (2011),
  pp.~1--32.

\bibitem{bollobas1984evolution}
{\sc B.~Bollob{\'a}s}, {\em The evolution of sparse graphs, in “graph theory
  and combinatorics proceedings}, in Cambridge Combinatorial Conference in
  Honour of Paul Erd\H{o}s, 1984, pp.~335--357.

\bibitem{conlon2016combinatorial}
{\sc D.~Conlon and W.~T. Gowers}, {\em Combinatorial theorems in sparse random
  sets}, Annals of Mathematics. Second Series, 184 (2016), pp.~367--454.

\bibitem{dellamonica2008resilience}
{\sc D.~Dellamonica, Jr., Y.~Kohayakawa, M.~Marciniszyn, and A.~Steger}, {\em
  On the resilience of long cycles in random graphs}, Electronic Journal of
  Combinatorics, 15 (2008).
\newblock Research Paper 32, 26 pages.

\bibitem{dudek2017some}
{\sc A.~Dudek and P.~Pra{\l}at}, {\em On some multicolor {R}amsey properties of
  random graphs}, SIAM Journal on Discrete Mathematics, 31 (2017),
  pp.~2079--2092.

\bibitem{DP2018}
{\sc A.~Dudek and P.~Pra\l{}at}, {\em Note on the multicolour size-{R}amsey
  number for paths}, Electronic Journal of Combinatorics, 25 (2018).
\newblock Research Paper 3.35, 5 pages.

\bibitem{erdHos1978size}
{\sc P.~Erd{\H{o}}s, R.~J. Faudree, C.~C. Rousseau, and R.~H. Schelp}, {\em The
  size {R}amsey number}, Periodica Mathematica Hungarica, 9 (1978),
  pp.~145--161.

\bibitem{EG1959}
{\sc P.~Erd{\H{o}}s and T.~Gallai}, {\em On maximal paths and circuits of
  graphs}, Acta Mathematica Hungarica, 10 (1959), pp.~337--356.

\bibitem{frankl1986large}
{\sc P.~Frankl and V.~R{\"o}dl}, {\em Large triangle-free subgraphs in graphs
  without ${K}_4$}, Graphs and Combinatorics, 2 (1986), pp.~135--144.

\bibitem{frieze1986large}
{\sc A.~M. Frieze}, {\em On large matchings and cycles in sparse random
  graphs}, Discrete Mathematics, 59 (1986), pp.~243--256.

\bibitem{gyarfas1977partition}
{\sc A.~Gy{\'a}rf{\'a}s}, {\em Partition coverings and blocking sets in
  hypergraphs (in {H}ungarian)}, Communications of the Computer and Automation
  Institute of the Hungarian Academy of Sciences, 71 (1977), p.~62.

\bibitem{haxell1995turan}
{\sc P.~E. Haxell, Y.~Kohayakawa, and T.~\L{}uczak}, {\em Tur{\'a}n's extremal
  problem in random graphs: Forbidding even cycles}, Journal of Combinatorial
  Theory, Series B, 64 (1995), pp.~273--287.

\bibitem{haxell1996turan}
{\sc P.~E. Haxell, Y.~Kohayakawa, and T.~\L{}uczak}, {\em Tur{\'a}n's extremal
  problem in random graphs: Forbidding odd cycles}, Combinatorica, 16 (1996),
  pp.~107--122.

\bibitem{johansson2008factors}
{\sc A.~Johansson, J.~Kahn, and V.~Vu}, {\em Factors in random graphs}, Random
  Structures \& Algorithms, 33 (2008), pp.~1--28.

\bibitem{kohayakawa1998extremal}
{\sc Y.~Kohayakawa, B.~Kreuter, and A.~Steger}, {\em An extremal problem for
  random graphs and the number of graphs with large even-girth}, Combinatorica,
  18 (1998), pp.~101--120.

\bibitem{kohayakawa1997onk}
{\sc Y.~Kohayakawa, T.~{\L}uczak, and V.~R{\"o}dl}, {\em On ${K}_4$-free
  subgraphs of random graphs}, Combinatorica, 17 (1997), pp.~173--213.

\bibitem{kohayakawa2004turan}
{\sc Y.~Kohayakawa, V.~R{\"o}dl, and M.~Schacht}, {\em The {T}ur{\'a}n theorem
  for random graphs}, Combinatorics, Probability and Computing, 13 (2004),
  pp.~61--91.

\bibitem{komlos1983limit}
{\sc J.~Koml{\'o}s and E.~Szemer{\'e}di}, {\em Limit distribution for the
  existence of {H}amiltonian cycles in a random graph}, Discrete Mathematics,
  43 (1983), pp.~55--63.

\bibitem{krivelevich2019expanders}
{\sc M.~Krivelevich}, {\em Expanders---how to find them, and what to find in
  them}, in Surveys in combinatorics 2019, vol.~456 of London Math. Soc.
  Lecture Note Ser., Cambridge Univ. Press, Cambridge, 2019, pp.~115--142.

\bibitem{krivelevich2019long}
\leavevmode\vrule height 2pt depth -1.6pt width 23pt, {\em Long cycles in
  locally expanding graphs, with applications}, Combinatorica, 39 (2019),
  pp.~135--151.

\bibitem{Krivelevich2019turan}
{\sc M.~Krivelevich, G.~Kronenberg, and A.~Mond}, {\em Tur{\'a}n-type problems
  for long cycles in random and pseudo-random graphs}, arXiv preprint
  arXiv:1911.08539,  (2019).

\bibitem{pikhurko2012}
{\sc O.~Pikhurko}, {\em A note on the {T}ur\'{a}n function of even cycles},
  Proceedings of the American Mathematical Society, 140 (2012), pp.~3687--3692.

\bibitem{ST}
{\sc D.~Saxton and A.~Thomason}, {\em Hypergraph containers}, Inventiones
  mathematicae, 201 (2015), pp.~925--992.

\bibitem{schacht2016extremal}
{\sc M.~Schacht}, {\em Extremal results for random discrete structures}, Annals
  of Mathematics,  (2016), pp.~333--365.

\bibitem{spohel2010coloring}
{\sc R.~Sp\"{o}hel, A.~Steger, and H.~Thomas}, {\em Coloring the edges of a
  random graph without a monochromatic giant component}, Electronic Journal of
  Combinatorics, 17 (2010).
\newblock Research Paper 133, 7 pages.

\bibitem{szabo2003turan}
{\sc T.~Szab{\'o} and V.~H. Vu}, {\em Tur{\'a}n's theorem in sparse random
  graphs}, Random Structures \& Algorithms, 23 (2003), pp.~225--234.

\end{thebibliography}
\end{document}